\documentclass[11pt]{amsart}

\usepackage{amsmath,amssymb,amsthm}
\usepackage{graphicx}
\usepackage{graphics}
\usepackage{xypic}
\usepackage[bookmarks=false]{hyperref} 
\usepackage{multirow}
\usepackage{multicol}

\usepackage[T1]{fontenc}
\usepackage{enumerate}

\usepackage{subcaption}

\usepackage[all]{xy}
\xyoption{poly}

\newcommand{\node}[1]{ *+[Fo]{#1} }
\newcommand{\exG}[5]{ \left( \ \raisebox{0.4cm}{\xymatrix@C-5ex@R-4ex{\node{#1} & & \node{#2} \ar@{-}[ll] \\ & \node{#3}\ar@{-}[ul] & }} \ , \ \xymatrix@-2ex{ \node{#4} \ar@{}[r] & \node{#5} } \ \right)}

\newtheorem{mylemma}{Lemma}
\newtheorem{mytheorem}[mylemma]{Theorem}
\newtheorem{myconjecture}[mylemma]{Conjecture}

\theoremstyle{definition}

\newtheorem{mydefinition}[mylemma]{Definition}

\newcommand{\R}{\mathbb{R}}
\newcommand{\Q}{\mathbb{Q}}
\newcommand{\C}{\mathcal{C}}

\newcommand{\pmat}[1]{\begin{pmatrix} #1 \end{pmatrix}}

\begin{document}

\begin{abstract}
Consider a variant of the graph diameter of a polyhedron where each step in a walk between two vertices travels maximally in a circuit direction instead of along incident edges. Here circuit directions are non-trivial solutions to minimally-dependent subsystems of the presentation of the polyhedron. These can be understood as the set of all possible edge directions, including edges that may arise from translation of the facets.

It is appealing to consider a circuit analogue of the Hirsch conjecture for graph diameter, as suggested by Borgwardt et al.~\cite{circuitarxiv}. They ask whether the known counterexamples to the
Hirsch conjecture give rise to counterexamples for this relaxed notion of circuit diameter. We show that the most basic counterexample to the unbounded Hirsch conjecture, the Klee-Walkup polyhedron,
does have a circuit diameter that satisfies the Hirsch bound, regardless of representation. We also examine the circuit diameter of the bounded Klee-Walkup polytope.
\end{abstract}

\title{Circuit Diameter and Klee-Walkup Constructions}
\author{Tamon Stephen}
\author{Timothy Yusun}
\address{Department of Mathematics \\
Simon Fraser University \\
8888 University Drive \\
Burnaby, B.C. V5A 1S6\\
Canada
}
\date{\today}
\maketitle

\section{Introduction}\label{intro}

Given a convex polyhedron, its \emph{graph} (also called its \emph{skeleton}) is the undirected graph formed by its vertices and edges. The \emph{graph diameter} of a polyhedron is then defined to be
the graph diameter of its skeleton. This is an interesting quantity since it gives us a lower bound on the performance of the simplex method for linear programming. It turns out that most known
polyhedra have diameters at most the \emph{Hirsch bound} of $f - d$, where $f$ and $d$ are the number of facets and the dimension, respectively. The main exceptions are unbounded polyhedra based on
the Klee-Walkup example~\cite{kleewalkup} and non-Hirsch polytopes based on the constructions of Santos~\cite{santoscounterexample}.

Here we consider the circuit diameter, where instead of being restricted to walk along edges of a polyhedron, one can walk in the direction of any `potential edges' obtained by translating its facets.
In particular, the set of permissible directions at a vertex has no connection with edges incident to it; circuit steps can go through the interior of the polyhedron and end when the direction is
traversed as far as possible maintaining feasibility. Because of this, it is challenging to prove anything about the circuit diameter independent of how the polyhedron is realized. In this note, we
show that the original unbounded non-Hirsch polyhedron of Klee and Walkup~\cite{kleewalkup} does satisfy the Hirsch bound in this relaxed framework, independent of realization.

\section{Circuit Walks and Diameters}\label{sec:circ}
\subsection{Background and Definitions}
While the graph diameter of a polyhedron considers walks along its edges, the circuit diameter considers walks that use the \emph{circuits} of a polyhedron, defined as follows:

\begin{mydefinition}\label{circuits}
Given a polyhedron $$P = \{x \in \R^n : A^1x = b^1, A^2x \geq b^2\},$$ where $A^i \in \Q^{d_i \times n}$ and $b^i \in \R^{d_i}$ for $i = 1,2$, the \emph{circuits} or \emph{elementary vectors} $\C(A^1,A^2)$ of $A^1$ and $A^2$ are defined as the set of vectors $g \in \ker(A^1) \ \backslash \ \{0\}$ for which $A^2g$ is support-minimal in the set $\{A^2x : x \in \ker(A^1) \ \backslash \ \{0\}\}$, where $g$ is normalized to coprime integer components.
\end{mydefinition}

It turns out that the set $\C(A^1,A^2)$ consists of exactly the possible edge directions of $P$ for varying $b^1$ and $b^2$ \cite{sturmfels}. Moreover, at any non-optimal feasible point of the linear program $$\min \{c^Tx : A^1x = b^1, A^2x \geq b^2\},$$ an augmenting direction can always be found from the set $\C(A^1,A^2)$.

Now, for a polyhedron $P$ and the set of circuits $\C$ associated with $A^1$ and $A^2$, given two vertices $u$ and $v$ of $P$ define a \emph{circuit walk of length $k$} to be a sequence $u = y^0, \ldots, y^k = v$ with

\vspace{-0.2cm}
\begin{enumerate}
\item $y^i \in P$
\item $y^{i+1} - y^i = \alpha_ig^i$ for some $g^i \in \C$ and $\alpha_i > 0$
\item $y^i + \alpha g_i \not\in P$ for $\alpha > \alpha_i$
\end{enumerate}

for all $i = 0, 1, \ldots, k - 1$. Observe that edge walks from $u$ to $v$ are exactly those circuit walks where each pair $y^i, y^{i+1}$ are adjacent vertices of $P$. Hence by considering all
possible circuit directions at each point $y^i$, instead of directions corresponding to incident edges, we see that the circuit walks are generalizations of edge walks. The \emph{circuit distance}
from $u$ to $v$ is now defined as the length of the shortest circuit walk from $u$ to $v$, and the \emph{circuit diameter} of $P$ is the largest circuit distance between any two vertices of $P$. Observe that the circuit distance is a lower bound for the graph distance.

We emphasize that this notion of circuit distance is not symmetric as it is for the graph distance. The following example from \cite{circuitarxiv} illustrates this:

\begin{figure}[htbp]
\begin{center}
\includegraphics[width=0.65\textwidth]{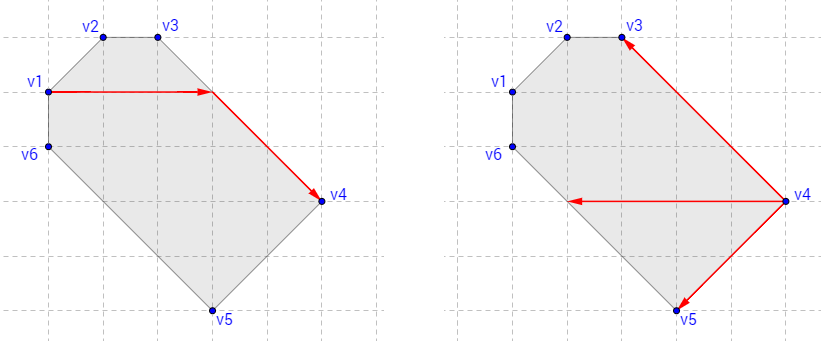}
\caption{The circuit distance from \texttt{v1} to \texttt{v4} is less than the circuit distance from \texttt{v4} to \texttt{v1}.}\label{fig1}
\end{center}
\end{figure}

Moreover, the circuit diameter of a polyhedron is dependent on its realization. This is evident by comparing the hexagon of Figure~\ref{fig1} to a regular hexagon, which has circuit diameter 2. Hence when we speak of `the circuit diameter' of $P$ it is implied that it is taken with respect to some representation of $P$. For more context on circuit diameter, see \cite{circuitarxiv} and \cite{hierarchy}.

\subsection{Variants of the Hirsch Conjecture}

There are polyhedra whose circuit diameter and graph diameter are the same -- a trivial example is the $d$-dimensional simplex, which has both graph and circuit diameter equal to 1. Also, take the $d$-dimensional cube with vertices all vectors in $\{0,1\}^d$. Its facet description is $$\{(x_1,x_2,\ldots,x_d) : 0 \leq x_i \leq 1, i = 1,2,\ldots, d\}.$$

This representation of the $d$-cube has circuits $\{\pm e_1, \pm e_2, \ldots, \pm e_d\}$, where $e_i$ is the vector with a $1$ in the $i$th position and $0$'s elsewhere. Hence its circuit diameter is also $d$. It is not clear if this can change in another representation of the $d$-cube.

The simplex and the cube are critical examples that motivated the well-known Hirsch conjecture. The conjecture can be stated as follows:

\begin{myconjecture}[Hirsch, 1957]
Let $f > d \geq 2$. Let $P$ be a $d$-dimensional polyhedron with $f$ facets. Then the combinatorial diameter of $P$ is at most $f - d$.
\end{myconjecture}

This is not true in general. Klee and Walkup in \cite{kleewalkup} found a counterexample that is an unbounded polyhedron in dimension 4, with 8 facets and diameter 5; this is featured in the next
section. The bounded case was finally settled by Santos \cite{santoscounterexample}, which has stimulated activity in this area. The conjecture does however hold for many interesting classes of
polyhedra.
See \cite{santossurvey} for a survey of recent research related to the Hirsch conjecture.

While the the Hirsch conjecture admits some hard-to-find exceptions for graph diameter,
the situation for circuit diameter is still unresolved.

\begin{myconjecture} \cite{circuitarxiv} \label{circuithirsch}
The circuit diameter of a $d$-dimensional polyhedron with $f$ facets is bounded above by $f-d$.
\end{myconjecture}

In Section~\ref{se:polyhedron}, we show that for circuit diameter, the most basic non-Hirsch unbounded polyhedron actually does satisfy the Hirsch bound, independent of representation. This provides some evidence for Conjecture~\ref{circuithirsch} by establishing it in one place where Hirsch does not hold for the graph diameter.

A facet can be added to the Klee-Walkup counterexample to obtain a 4-dimensional \emph{Klee-Walkup polytope} with 9 facets and diameter exactly 5. Polyhedra whose graph diameter meets the Hirsch bound
are called \emph{Hirsch-sharp}. Other basic examples of these are the $d$-simplex and the $d$-cube, and certain polytopes arising from operations performed on Hirsch-sharp polytopes; see
\cite{fritzscheholt} and \cite{holtklee}.  In Section~\ref{se:polytope}, we investigate the circuit diameter of the Klee-Walkup polytope, with a view to Hirsch-sharpness.

\section{The Klee-Walkup Polyhedron}\label{se:polyhedron}

The first unbounded non-Hirsch polyhedron was given by Klee and Walkup in \cite{kleewalkup}, where they constructed a 4-dimensional polyhedron with 8 facets and diameter 5. Its facet description is $\{x \in \R^4 : Ax \geq b\}$ where 

$$A = \begin{pmatrix} -6 & -3 & 0 & 1 \cr
-3 & -6 & 1 & 0 \cr
-35 & -45 & 6 & 3 \cr
-45 & -35 & 3 & 6 \cr
1 & 0 & 0 & 0 \cr
0 & 1 & 0 & 0 \cr
0 & 0 & 1 & 0 \cr
0 & 0 & 0 & 1  \end{pmatrix}
\text{ and }
b = \begin{pmatrix}
-1 \cr -1 \cr -8 \cr -8 \cr 0 \cr 0 \cr 0 \cr 0
\end{pmatrix}. $$

We call the combinatorial class of this polyhedron $U_4$, and this particular realization by $\tilde U_4$. Its vertex-edge graph is shown in Figure~\ref{uQ4graph}. Here the vertices are indexed by the four facets containing each one, while the points labelled with \texttt{R}'s represent extreme rays. It is clear from the graph that vertices \texttt{V5678} and \texttt{V1234} are at graph distance five apart.

\begin{figure}[h]
\begin{center}
\includegraphics[width=0.45\textwidth]{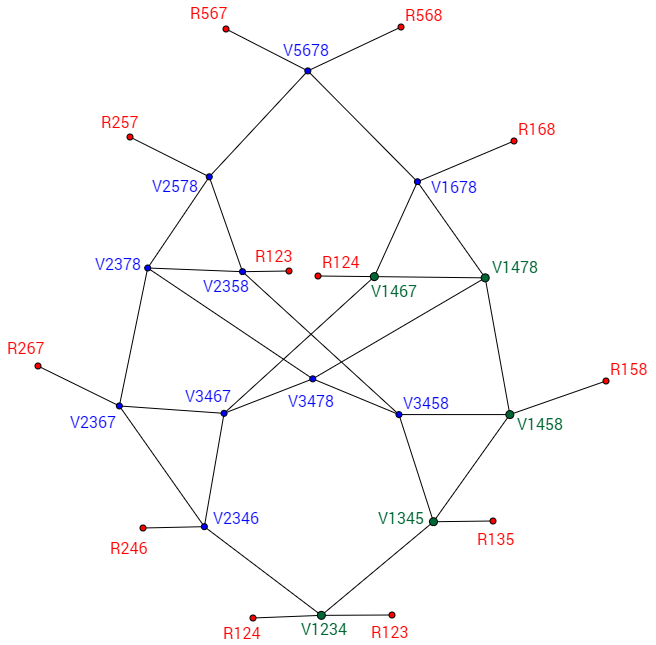}
\caption{The skeleton of $U_4$.}\label{uQ4graph}
\end{center}
\end{figure}

We now prove the following result:

\begin{mytheorem}\label{U4thm}
The circuit diameter of the Klee-Walkup polyhedron $U_4$ is at most 4, independent of representation.
\end{mytheorem}
\begin{proof}
First we demonstrate the existence of a circuit walk of length 4 from \texttt{V5678} to \texttt{V1234}. Observe that we can take two edge steps as follows: \texttt{V5678} $\rightarrow$ \texttt{V1678} $\rightarrow$ \texttt{V1478}. Vertices \texttt{V1478} and \texttt{V1234} are both contained in the 2-face determined by facets \texttt{1} and \texttt{4}, so we can complete the walk on this face. Note that this 2-face is an unbounded polyhedron on six facets. Figure~\ref{uQ4face14a} is a topological illustration of this face, showing the order of the vertices and rays.

\begin{figure}[htbp]
\begin{center}
\includegraphics[width=0.35\textwidth]{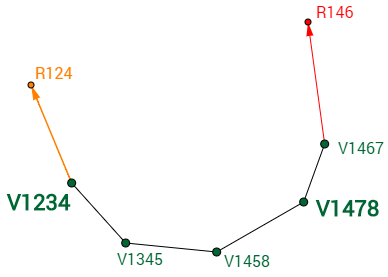}
\caption{The 2-face determined by facets \texttt{1} and \texttt{4}.} \label{uQ4face14a}
\end{center}
\end{figure}

Now consider a vector $\vec g$ corresponding to the edge direction from \texttt{V1458} to \texttt{V1345} -- this is the blue vector in Figure~\ref{uQ4face14a_feas}. Note that this is always a circuit direction in any representation of $U_4$ since it corresponds to an actual edge of the polyhedron.

To see that $\vec g$ is a feasible direction at \texttt{V1478}, consider vector $\vec h$ in the edge direction from \texttt{V1478} to \texttt{V1458}, and vector $\vec r$ in the direction of ray \texttt{R124}. Observe that $\vec g$ and $-\vec h$ are the two incident edge directions at \texttt{V1458}, and so $\vec r$ must be a strict conic combination of $\vec g$ and $-\vec h$, i.e. $\vec r = \alpha_1 (\vec g) + \alpha_2 (-\vec h)$ for $\alpha_1, \alpha_2 > 0$. By rearranging terms we see that $\vec g$ is a strict conic combination of $\vec h$ and $\vec r$: $\vec g = (\alpha_2/\alpha_1)\vec h + (1/\alpha_1)\vec r$, with $\alpha_2/\alpha_1, 1/\alpha_1 > 0$. Feasibility of $\vec r$ and $\vec h$ at \texttt{V1478} implies that $\vec g$ is a feasible direction at \texttt{V1478}.

\begin{figure}[htbp]
\begin{center}
\includegraphics[width=0.4\textwidth]{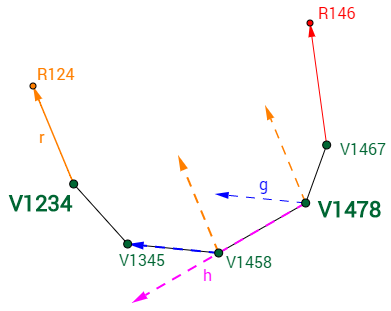}
\caption{Feasibility of the circuit direction $\vec g$.} \label{uQ4face14a_feas}
\end{center}
\end{figure}

Now starting at \texttt{V1478} traverse $\vec g$ as far as feasibility allows. This direction is bounded since we exit the polyhedron when taking $\vec g$ from \texttt{V1458}. We will eventually exit the 2-face at a point along the boundary, and at one of the following positions:

\begin{itemize}
\item exactly at \texttt{V1234},
\item on the edge connecting \texttt{V1234} and \texttt{V1345}, or
\item on the ray \texttt{R124} emanating from \texttt{V1234}.
\end{itemize}

Hitting exactly \texttt{V1234} gives a circuit walk of length 3 from \texttt{V5678}, while the other two cases give circuit walks of length 4 since we only need one step to \texttt{V1234}. These two situations are illustrated in Figure~\ref{uQ4face14}.

\begin{figure}[htbp]
\begin{center}
\includegraphics[width=0.85\textwidth]{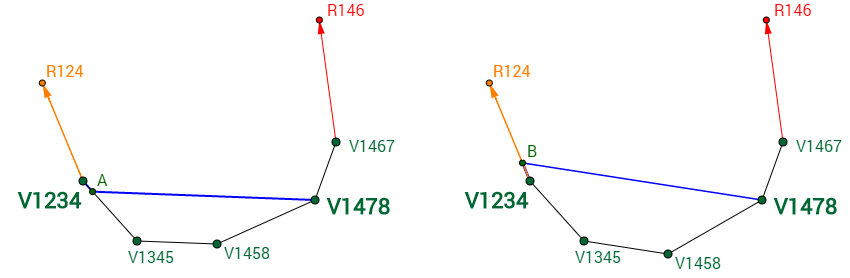}
\caption{Getting from \texttt{V1478} to \texttt{V1234} in at most 2 steps.} \label{uQ4face14}
\end{center}
\end{figure}

The argument is the same for the reverse direction (\texttt{V1234} to \texttt{V5678}). We can construct a similar walk by first traversing edges \texttt{V1234} $\rightarrow$ \texttt{V2346} $\rightarrow$ \texttt{V3467}, and then taking a maximal step in the circuit direction arising from the edge connecting \texttt{V1467} and \texttt{V1678}. Here we stay in the 2-face determined by facets \texttt{6} and \texttt{7}. We can then arrive at \texttt{V5678} in at most two steps from \texttt{V3467}. 
\end{proof}

This result illustrates that the first counterexample to the unbounded Hirsch conjecture does have a circuit diameter satisfying the Hirsch bound. We remark that we computed the circuit diameter for $\tilde U_4$ using brute force and it was 4. We do not know if it could be lower in some other representation. Our computational strategy is outlined in Section~\ref{app:comp}.

\section{The Klee-Walkup Polytope} \label{se:polytope}

By applying a projective transformation to $U_4$ we can obtain a bounded 4-dimensional polytope with $f = 9$ and diameter 5. As in \cite{kimsantos}, we call this polytope $Q_4$. Since its diameter is exactly $5 = f - d = 9 - 4$, $Q_4$ is \emph{Hirsch-sharp} -- this is the smallest Hirsch-sharp polytope outside of the classes mentioned in the previous section. With this in mind, we investigate the circuit diameter of $Q_4$. 

\begin{figure}[htbp]
\begin{center}
\includegraphics[width=0.8\textwidth]{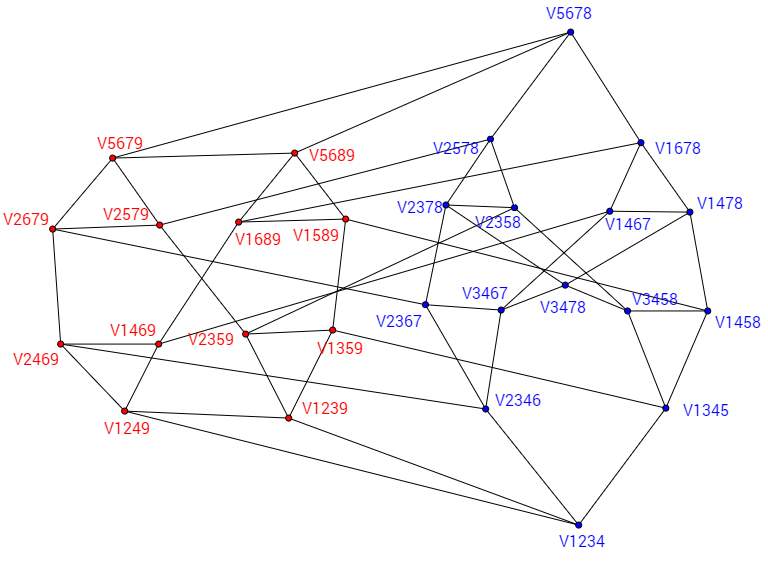}
\caption{The skeleton of $Q_4$.} \label{Q4skel}
\end{center}
\end{figure}

The graph of $Q_4$ is shown in Figure~\ref{Q4skel}. Here the red vertices are the ones obtained by adding the ninth facet to $U_4$. It is easy to see that $Q_4$ has graph diameter 5: starting at \texttt{V5678} and using only blue vertices, we  stay in $U_4$, and \texttt{V1234} is at distance 5; if we move to a red vertex at any point, we will need at least 4 other steps to introduce facets \texttt{1}, \texttt{2}, \texttt{3}, and \texttt{4}.

The arguments given in the proof of Theorem~\ref{U4thm} cannot be applied to $Q_4$ because we lose unboundedness. That is, in the 2-face determined by facets \texttt{1} and \texttt{4} (which is now a
bounded 7-gon), we cannot anymore be certain that we can get to \texttt{V1234} in two steps  from \texttt{V1478} (or \texttt{V1467}, which is at distance 2 from \texttt{V5678}). Figure~\ref{7gon}
illustrates a situation where we would need three steps to arrive at \texttt{V1234} coming from either \texttt{V1478} or \texttt{V1467}.  
It is not clear if $Q_4$ can be realized with the 2-face as illustrated.  In general, it is not always possible to prescribe the
shape of a 2-face when realizing a 4-polytope~\cite{realization}.

\begin{figure}[htbp]
\begin{center}
\includegraphics[width=0.45\textwidth]{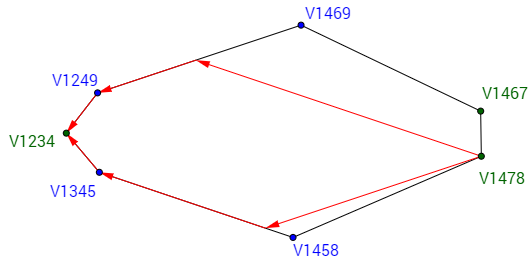}
\caption{At least three steps are needed to go from \texttt{V1478} to \texttt{V1234}.} \label{7gon}
\end{center}
\end{figure}

It is interesting to note, however, that the non-existence of a 2-walk from \texttt{V1478} to \texttt{V1234} in this face implies the existence of a 2-walk from \texttt{V1234} to \texttt{V1478}, and hence a circuit walk of length 4 from \texttt{V1234} to \texttt{V5678} (see Figure~\ref{7gon2}). As a consequence we are able to attain a walk of distance 4 between \texttt{V1234} and \texttt{V5678} \emph{for at least one direction}, irrespective of representation.
This shows that for certain alternative notions of circuit diameter presented by Borgwardt et al.~\cite{hierarchy}, $Q_4$ is no longer Hirsch-sharp.

\begin{figure}[htbp]
\begin{center}
\includegraphics[width=0.45\textwidth]{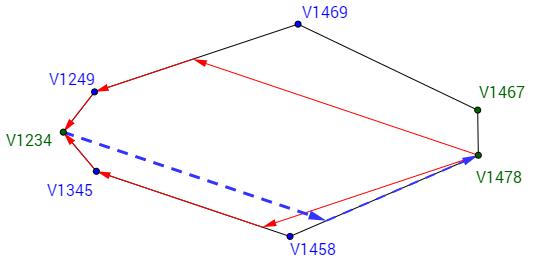}
\caption{From \texttt{V1234} to \texttt{V1478} in two steps.} \label{7gon2}
\end{center}
\end{figure}

We speculate that no representation of $Q_4$ is circuit Hirsch-sharp, i.e.~that the circuit diameter is at most 4. Below we show one representation of $Q_4$ that has circuit diameter 3, and another that has circuit diameter 4.

\subsection{A Representation of $Q_4$ with Circuit Diameter 3}\label{app:niceq4}

We first consider the representation of $Q_4$, given by the facet description $\tilde Q_4 = \{x \in \R^4 : Ax \geq b\}$ from \cite{kimsantoscomp}. This representation has several symmetries in its coefficients:

$$
A = \begin{pmatrix} 3 & -3 & -1 & -2 \cr
-3 & 3 & -1 & -2 \cr
-2 & 1 & -1 & -3 \cr
2 & -1 & -1 & -3 \cr
-3 & -3 & 1 & -2 \cr
3 & 3 & 1 & -2 \cr
1 & 2 & 1 & -3 \cr
-1 & -2 & 1 & -3 \cr
0 & 0 & 0 & 2 \cr
\end{pmatrix}
\text{ and }
b = \begin{pmatrix} -1 \cr -1 \cr -1 \cr -1 \cr -1 \cr -1 \cr -1 \cr -1 \cr -1 \end{pmatrix}.
$$

Its $f$-vector is $(27,54,36,9)$, and its graph diameter is 5 (refer to Figure~\ref{Q4skel} for its graph). Other than \texttt{V1234} and \texttt{V5678}, all other pairs of vertices are connected by paths of lengths at most 4.

By brute force enumeration, we are able to generate a circuit walk of length 2 from vertex \texttt{3} to vertex \texttt{15}, that uses the circuits $g^0 = (-1,0,-3,0)^T$ and $g^1 = (1,0,-3,0)^T$. This shows in fact that the circuit diameter of $\tilde Q_4$ is less than 5.

Further, by an exhaustive search we find circuit walks of length 3 between every other pair of vertices. Some of these pairs admit a walk of length 2 while some pairs do not. Hence $\tilde Q_4$ has circuit diameter 3. See Appendix~\ref{app:comp} for more details of this computation.

\subsection{A Representation of $Q_4$ with Circuit Diameter 4}\label{app:q4diam3}

We are able to use the same methods to explore the circuit diameters of perturbed, less symmetric versions of $Q_4$, with the goal of finding representations with different diameters. For instance, consider the realization $\tilde Q_4' = \{x \in \R^4: Cx \geq b\}$, where $$C = \pmat{ 3.2 & -3 & -1 & -2 \cr
-3 & 3.2 & -1 & -2 \cr
-2 & 1 & -1 & -3 \cr
2 & -1 & -1 & -3 \cr
-3 & -3 & 1.05 & -2 \cr
3 & 3 & 1.05 & -2 \cr
1.05 & 2 & 1 & -3 \cr
-1 & -2.05 & 1 & -3 \cr
0 & 0 & 0 & 2} = A + \pmat{0.2 & 0 & 0 & 0 \cr
0 & 0.2 & 0 & 0 \cr
0 & 0 & 0 & 0 \cr
0 & 0 & 0 & 0 \cr
0 & 0 & 0.05 & 0 \cr
0 & 0 & 0.05 & 0 \cr
0.05 & 0 & 0 & 0 \cr
0 & -0.05 & 0 & 0 \cr
0 & 0 & 0 & 0 },$$

and $A$ and $b$ are as defined in the standard representation. (A quick check is done in \texttt{polymake} to verify the isomorphism and generate the vertices.) Vertices \texttt{V5678} and \texttt{V1234} of $\tilde Q_4'$, while still separated by a distance-5 path along the skeleton of the polytope, are now at circuit distance 4 apart. Therefore $\tilde Q_4'$ has circuit diameter exactly $4$, as there are multiple pairs of vertices separated by circuit distance 4.  This brings to light the fragility of the circuit diameter with respect to geometric realization.

\section{Computational Details}\label{app:comp}

We performed the circuit diameter computations using \texttt{polymake} (\cite{polymake}), MATLAB (\cite{matlab}), and Maple (\cite{maple}). Given the representation $\{x \in \R^4 : Ax \geq b\}$ of
$Q_4$, we first use \texttt{polymake} to compute its vertices, then we compute circuit directions in MATLAB. Since each $4 \times 4$ submatrix of $A$ is nonsingular, we are able to compute all vectors
$g$ such that $Ag$ is support minimal by solving the systems $A_{ijk}g = 0$ for $1 \leq i < j < k \leq 9$, where $A_{ijk}$ is the $3 \times 4$ submatrix of $A$ consisting of rows $i$, $j$, and $k$.
Each $A_{ijk}$ matrix is rank 3 and so the solution to $A_{ijk}g = 0$ is a line passing through the origin. Normalizing to coprime integer components gives the required circuit direction $g$. Any circuit step can then be computed using a function that, given a starting point $x$ and a circuit direction $g$ as inputs, finds $\alpha_g$ so that $A(x + \alpha_g g) \geq b$ but $A(x + \alpha g) \not\geq b$ for $\alpha > \alpha_g$.

We then enumerate exhaustively all circuit walks of length 2 or 3 emanating from any point, by considering all possible triples of circuit directions. With $Q_4$, we have 84 circuit directions, and taking into account both the positive and negative of each one, we find that there are at most $168\cdot (166)^2 = 4629408$ triples to check. The MATLAB code can perform this check in a few minutes. Afterwards, to check for circuit walks of length 4, we use the set of points output by the above enumeration procedure and compute circuit steps for each feasible circuit. Again, this does not take long to run. We perform checks by hand or using Maple since some precision is lost when using MATLAB for exact arithmetic. All these computations were performed on a laptop in a few minutes.

Using the above method, we found that $\tilde  U_4$ has circuit diameter 4, $\tilde Q_4$ has circuit diameter 3, and $\tilde Q_4'$ has circuit diameter 4.

\section{Acknowledgements} This research was partially supported by an NSERC Discovery Grant for T. Stephen, and an NSERC Postgraduate Scholarship-D for T. Yusun. All illustrations were produced using the GeoGebra software\footnote{http://www.geogebra.org, International GeoGebra Institute.}.

\bibliographystyle{halpha}
\bibliography{referencesB}

\end{document}